\documentclass[12pt]{article}

\usepackage{amssymb}
\usepackage{amsmath}
\usepackage[usenames]{color}
\usepackage{graphicx}
\usepackage{amsthm}
\usepackage{graphics}
\usepackage[colorlinks=true,
linkcolor=webgreen,
filecolor=webbrown,
citecolor=webgreen]{hyperref}

\definecolor{webgreen}{rgb}{0,.5,0}
\definecolor{webbrown}{rgb}{.6,0,0}
\definecolor{red}{rgb}{1,0,0}

\usepackage{color}
\usepackage{url}

\usepackage{amsfonts}
\usepackage{latexsym}

\theoremstyle{plain}
\newtheorem{theorem}{Theorem}

\newtheorem{lemma}[theorem]{Lemma}

\theoremstyle{definition}
\newtheorem{definition}[theorem]{Definition}

\theoremstyle{remark}

\begin{document}

\title{Binomial transforms of the modified $k$-Fibonacci-like sequence}

\date{}
\author{Youngwoo Kwon\\
Department of mathematics, Korea University, Seoul, Republic of Korea\\
\href{mailto:ywkwon81@korea.ac.kr}{\tt ywkwon81@korea.ac.kr}\\
}

\maketitle
\begin{abstract}
This study applies the binomial, $k$-binomial, rising $k$-binomial and falling $k$-binomial transforms to the modified $k$-Fibonacci-like sequence. Also, the Binet formulas and generating functions of the above mentioned four transforms are newly found by the recurrence relations.
\end{abstract}

\section{Introduction}
The Fibonacci sequence $\left(F_{n}\right)_{n\geq0}$ is defined by the recurrence relation
\begin{align*}
F_{n+1}=&F_{n}+F_{n-1} \text{~for~} n\ge1
\end{align*}
with the initial conditions $F_{0}=0$ and $F_{1}=1$.

Many authors have studied the Fibonacci sequence, some of whom introduced new sequences related to it as well as proving many identities for them.

In particular, Falc$\acute{\rm{o}}$n and Plaza \cite{FP02} introduced the $k$-Fibonacci sequence.

\begin{definition}[\cite{FP02}]
For any positive real number $k$, the $k$-Fibonacci sequence $\left( F_{k,n} \right)_{n\geq0}$ is defined by recurrence relation
$$F_{k,n+1} = k F_{k,n} + F_{k,n-1} ~\text{for}~n\ge1$$
with the initial conditions $F_{k,0} =0$ and  $F_{k,1} =1$.
\end{definition}

Also, Kwon \cite{YK} introduced the modified $k$-Fibonacci-like sequence.
\begin{definition}[\cite{YK}]
For any positive real number $k$, the modified $k$-Fibonacci-like sequence $\left(M_{k,n}\right)_{n\geq0}$ is defined by the recurrence relation
$$M_{k,n+1} = k M_{k,n} + M_{k,n-1} ~\text{for}~n\ge1 $$
with the initial conditions $M_{k,0} = M_{k,1} = 2$.
\end{definition}

The first few modified $k$-Fibonacci-like numbers are as follows:
\begin{align*}
M_{k,2}=& 2k+2,\\
M_{k,3}=& 2k^{2}+2k+2,\\
M_{k,4}=& 2k^{3}+2k^{2}+4k+2,\\
M_{k,5}=& 2k^{4}+2k^{3} +6k^{2}+4k+2.
\end{align*}


Kwon \cite{YK} studied the following identities between the $k$-Fibonacci sequence and the modified $k$-Fibonacci-like sequence.
$$M_{k,n} = 2 \left( F_{k,n} + F_{k,n-1}\right) \text{ and } F_{k,n} = \frac{1}{2}\sum_{i=0}^{n-1} M_{k,n-i}(-1)^{i}$$

Spivey and Steil \cite{SS} introduced various binomial transforms.
\begin{enumerate}
\item[(1)] The binomial transform $B$ of the integer sequence $A=\left\{a_{0}, a_1 , a_2 , \ldots \right\}$, which is denoted by $B(A)=\left\{b_{n}\right\}$ and defined by
$$b_{n} = \sum_{i=0}^{n}\binom{n}{i} a_{i}.$$
\item[(2)] The $k$-binomial transform $W$ of the integer sequence $A=\left\{a_{0}, a_1 , a_2 , \ldots \right\}$, which is denoted by $W(A)=\left\{w_{n}\right\}$ and defined by
$$w_{n} = \sum_{i=0}^{n}\binom{n}{i} k^{n} a_{i}.$$
\item[(3)] The rising $k$-binomial transform $R$ of the integer sequence $A=\left\{a_{0}, a_1 , a_2 , \ldots \right\}$, which is denoted by $B(A)=\left\{r_{n}\right\}$ and defined by
$$r_{n} = \sum_{i=0}^{n}\binom{n}{i} k^{i} a_{i} .$$
\item[(4)] The falling $k$-binomial transform $F$ of the integer sequence $A=\left\{a_{0}, a_1 , a_2 , \ldots \right\}$, which is denoted by $F(A)=\left\{f_{n}\right\}$ and defined by
$$f_{n} = \sum_{i=0}^{n}\binom{n}{i} k^{n-i} a_{i} .$$
\end{enumerate}

Other latest research \cite{BJS, FP03, YT02} also examined the various binomial transforms for several special sequences. These transforms are interesting and meaningful as they introduced several new approaches.

Based on those preceding studies, this study applies the four binomial transforms namely, binomial, $k$-binomial, rising $k$-binomial and falling $k$-binomial transforms to the modified $k$-Fibonacci-like sequence. This study also proves their properties.

\section{The binomial transform of the modified \texorpdfstring{$k$}{Lg}-Fibonacci-like sequence}
The binomial transform of the modified $k$-Fibonacci-like sequence $\left(M_{k,n}\right)_{n\geq0}$ is denoted by $B_{k}=\left(b_{k,n}\right)_{n \geq0}$ where
$$b_{k,n} = \sum_{i=0}^{n}\binom{n}{i}M_{k,i}.$$

The only binomial transforms of the modified $k$-Fibonacci-like sequences indexed in OEIS \cite{S} are as follows:
\begin{align*}
B_{1} =&\left\{2, 4, 10, 26, 68, 178, \ldots \right\} : A052995-\{0\} \text{ or } A055819-\{1\}\\
B_{2} =&\left\{2, 4, 12, 40, 136, 464, \ldots \right\} : A056236\\
B_{3} =&\left\{2, 4, 14, 58, 248, 1066, \ldots \right\} \\
B_{4} =&\left\{2, 4, 16, 80, 416, 2176, \ldots \right\} \\
B_{5} =&\left\{2, 4, 18, 106, 652, 4034, \ldots \right\}
\end{align*}

\begin{lemma}\label{binomial_T}
The binomial transform of the modified $k$-Fibonacci-like sequence satisfies the relation
$$b_{k,n+1} - b_{k,n} = \sum_{i=0}^{n}\binom{n}{i} M_{k,i+1}.$$
\end{lemma}
\begin{proof}
Note that $\binom{n}{0}=1$ and $\binom{n+1}{i} = \binom{n}{i}+\binom{n}{i-1}$.

The difference of the two consecutive binomial transforms is the following:
\begin{align*}
b_{k,n+1} -b_{k,n}&= \sum_{i=0}^{n+1} \binom{n+1}{i} M_{k,i}-\sum_{i=0}^{n}\binom{n}{i}M_{k,i}\\
&=\sum_{i=1}^{n}\left[\binom{n+1}{i}-\binom{n}{i} \right]M_{k,i} + M_{k,n+1} \\
&=\sum_{i=1}^{n}\binom{n}{i-1}M_{k,i}+M_{k,n+1}\\
&=\sum_{i=0}^{n-1}\binom{n}{i}M_{k,i+1} + \binom{n}{n}M_{k,n+1}=\sum_{i=0}^{n}\binom{n}{i}M_{k,i+1}
\end{align*}
\end{proof}
Note that $b_{k,n+1} = \sum_{i=0}^{n}\binom{n}{i}\left(M_{k,i}+M_{k,i+1}\right)$.
\begin{theorem}\label{binomial_Ta}
The binomial transform of the modified $k$-Fibonacci-like sequence $B_{k}=\left(b_{k,n}\right)_{n\geq0}$ satisfies the recurrence relation
$$b_{k,n+1} = (k+2)b_{k,n} - k b_{k,n-1} ~\text{for}~n\ge1$$
with the initial conditions $b_{k,0}=2$, $b_{k,1} = 4$.
\end{theorem}
\begin{proof}
By Lemma \ref{binomial_T}, since $b_{k,n+1} = \sum_{i=0}^{n}\binom{n}{i}\left(M_{k,i}+M_{k,i+1}\right)$, then we have
\begin{align*}
b_{k,n+1}=&M_{k,0} +M_{k,1} + \sum_{i=1}^{n}\binom{n}{i}\left(M_{k,i} + M_{k,i+1}\right)\\
=&M_{k,0}+M_{k,1}+\sum_{i=1}^{n}\binom{n}{i}\left(M_{k,i}+kM_{k,i}+M_{k,i-1}\right)\\
=&\left[(k+1)M_{k,0}+(k+1)\sum_{i=1}^{n}\binom{n}{i}M_{k,i}\right]\\
&+\sum_{i=1}^{n}\binom{n}{i}M_{k,i-1} + M_{k,1}-kM_{k,0}\\
=&(k+1)\sum_{i=0}^{n}\binom{n}{i}M_{k,i}+\sum_{i=1}^{n}\binom{n}{i}M_{k,i-1}+M_{k,1}-kM_{k,0}\\
=&(k+1)b_{k,n}+\sum_{i=1}^{n}\binom{n}{i}M_{k,i-1} +2-2k.
\end{align*}
On the other hand, in the case of $\binom{n-1}{n}=0$, we can obtain the following:
\begin{align*}
b_{k,n} &=kb_{k,n-1} + \sum_{i=1}^{n}\binom{n}{i}M_{k,i-1}+2-2k.
\end{align*}
Based on the above two identities, this study draws the below formulas.
$$b_{k,n+1} - (k+1)b_{k,n} = b_{k,n} - k b_{k,n-1},$$
and so
$$b_{k,n+1} = (k+2)b_{k,n} - k b_{k,n-1}.$$
\end{proof}

Binet's formulas are well known in the Fibonacci number theory. In this study, Binet's formula for the binomial transform of the modified $k$-Fibonacci-like sequence is suggested as the following:

\begin{theorem}\label{binet_T}
Binet's formula for the binomial transform of the modified $k$-Fibonacci-like sequence is given by
$$b_{k,n}= 4\frac{r_{1}^{n}-r_{2}^{n}}{r_{1}-r_{2}}- 2k\frac{r_{1}^{n-1}-r_{2}^{n-1}}{r_{1}-r_{2}},$$
where $r_{1}$ and $r_{2}$ are the roots of the characteristic equation $x^{2}-(k+2)x+k=0$, and $r_{1}>r_{2}$.
\end{theorem}
\begin{proof}
The characteristic polynomial equation of $b_{k,n+1} = (k+2)b_{k,n} - k b_{k,n-1}$ is $x^{2} - (k+2)x + k =0$, whose solution are $r_{1}$ and $r_{2}$ with $r_{1}>r_{2}$. The general term of the binomial transform may be expressed in the form, $b_{k,n}=C_{1}r_{1}^{n} + C_{2}r_{2}^{n}$ for some coefficients $C_{1}$ and $C_{2}$.
\begin{enumerate}
\item[(1)] $b_{k,0} = C_{1} + C_{2} = 2$
\item[(2)] $b_{k,1}=C_{1}r_{1} + C_{2}r_{2} = 4$
\end{enumerate}
Then
$$C_{1} = \frac{4-2r_{2}}{r_{1}-r_{2}} \text{ and } C_{2}=\frac{2r_{1} -4}{r_{1}-r_{2}}.$$
Therefore,
$$b_{k,n} = \frac{4-2r_{2}}{r_{1}-r_{2}} r_{1}^{n} + \frac{2r_{1} -4}{r_{1}-r_{2}} r_{2}^{n} =4 \frac{r_{1}^{n}-r_{2}^{n}}{r_{1}-r_{2}} -2k\frac{r_{1}^{n-1}-r_{2}^{n-1}}{r_{1}-r_{2}}. $$
\end{proof}

The binomial transform $B_{k}$ can be seen as the coefficients of the power series which is called the generating function. Therefore, if $b_{k}(x)$ is the generating function, then we can write
$$b_{k}(x)=\sum_{i=0}^{\infty} b_{k,i}x^{i} = b_{k,0}+b_{k,1}x+b_{k,2}x^{2}+\cdots.$$
And then,
\begin{align*}
(k+2)x b_{k}(x)=&(k+2)b_{k,0}x+(k+2)b_{k,1}x^2 + (k+2)b_{k,2}x^{3}+\cdots,\\
kx^{2}b_{k}(x)=&k b_{k,0}x^{2} + k b_{k,1}x^{3} + k b_{k,2} x^{4} + \cdots.
\end{align*}
Since $b_{k,n+1} - (k+2)b_{k,n} + k b_{k,n-1} = 0$, $b_{k,0}=2$, and $b_{k,1}=4$, then we have
\begin{align*}
&(1-(k+2)x+kx^{2}) b_{k} (x)\\
=& b_{k,0} + (b_{k,1} - (k+2)b_{k,0})x + (b_{k,2} - (k+2)b_{k,1} + kb_{k,0} )x^{2} + \cdots\\
=& b_{k,0} + (b_{k,1}-(k+2)b_{k,0})x\\
=&2+(4-(k+2)2)x = 2-2kx.
\end{align*}
Hence, the generating function for the binomial transform of the modified $k$-Fibonacci-like sequence $\left(b_{k,n}\right)_{n\geq0}$ is
$$b_{k}(x) = \frac{2(1-2kx)}{1-(k+2)x+kx^{2}}.$$

\section{The \texorpdfstring{$k$}{Lg}-binomial transform of the modified \texorpdfstring{$k$}{Lg}-Fibonacci-like sequence }
The $k$-binomial transform of the modified $k$-Fibonacci-like sequence $\left(M_{k,n}\right)_{n\geq0}$ is denoted by $W_{k}=\left(w_{k,n}\right)_{n\geq0}$ where
\begin{displaymath}
w_{k,n} =
\begin{cases}
\sum_{i=0}^{n}\binom{n}{i} k^{n}M_{k,i}, & \text{ for }  k\ne 0  \text{ or }  n\ne 0; \\
0, &  \text{ if }  k=0 \text{ and } n=0.
\end{cases}
\end{displaymath}
The first $k$-binomial transforms are as follows:
\begin{align*}
W_{1}=&\left\{2, 4, 10, 26, 68, 178, \ldots \right\} : A052995-\{0\} \text{ or } A055819-\{1\}\\
W_{2}=&\left\{2, 8, 96, 320, 1088, 3712, \ldots \right\}\\
W_{3}=&\left\{2, 12, 378, 1566, 6696, 28782, \ldots \right\}\\
W_{4}=&\left\{2, 16, 1024, 5120, 26624, \ldots \right\}\\
W_{5}=&\left\{2, 20, 2250, 13250, 81500, \ldots \right\}
\end{align*}

Note that the $1$-binomial transform $W_{1}$ coincides with the binomial transform $B_{1}$.

Note that
$$w_{k,n} = \sum_{i=0}^{n}\binom{n}{i}k^{n} M_{k,i} = k^{n} \sum_{i=0}^{n}\binom{n}{i} M_{k,i}= k^{n}b_{k,n},$$
$$\text{and so } w_{k,n+1} = k^{n+1}\sum_{i=0}^{n}\binom{n}{i}\left(M_{k,i} + M_{k,i+1}\right)$$
from Lemma \ref{binomial_T}
\begin{theorem}
The $k$-binomial transform of the modified $k$-Fibonacci-like sequence $W_{k}=\left(w_{k,n}\right)_{n\geq0}$ satisfies the recurrence relation
$$w_{k,n+1}  =k(k+2)w_{k,n} - k^{3} w_{k,n-1} ~\text{for}~n\ge1$$
with the initial conditions $w_{k,0} = 2$, $w_{k,1} = 4k$.
\end{theorem}
\begin{proof}
By Theorem \ref{binomial_Ta}, we can easily obtain the following:
\begin{align*}
w_{k,n+1}&=k^{n+1} b_{k,n+1}\\
&=k^{n+1} \left[ (k+2)b_{k,n}-k b_{k,n-1}\right]\\
&=k^{n+1}(k+2)b_{k,n} - k^{n+2} b_{k,n-1}\\
&=k(k+2)w_{k,n} - k^{3} w_{k,n-1}
\end{align*}
\end{proof}

Similarly, Binet's formula for the $k$-binomial transform of the modified $k$-Fibonacci-like sequence is the following:
\begin{theorem}
Binet's formula for the $k$-binomial transform of the modified $k$-Fibonacci-like sequence is given by
$$w_{k,n}= 4\frac{s_{1}^{n}-s_{2}^{n}}{s_{1}-s_{2}}- 2k\frac{s_{1}^{n-1}-s_{2}^{n-1}}{s_{1}-s_{2}},$$
where $s_{1}$ and $s_{2}$ are the roots of the characteristic equation $x^{2}-k(k+2)x+k^{3}=0$, and $s_{1}>s_{2}$.
\end{theorem}
\begin{proof}
The proof is same as that of the binomial transform, which is in Theorem \ref{binet_T}.
\end{proof}

Similarly, the generating function for the $k$-binomial transform of the modified $k$-Fibonacci-like sequence is
$$w_{k}(x) = \frac{2(1-k^{2}x)}{1-k(k+2)x+k^{3}x^{2}}.$$

\section{The rising \texorpdfstring{$k$}{Lg}-binomial transform of the modified \texorpdfstring{$k$}{Lg}-Fibonacci-like sequence }
The rising $k$-binomial transform of the modified $k$-Fibonacci-like sequence $\left(M_{k,n}\right)_{n\geq0}$ is denoted by $R_{k}=\left(r_{k,n}\right)_{n\geq0}$ where
\begin{displaymath}
r_{k,n} = \begin{cases}
\sum_{i=0}^{n}\binom{n}{i} k^{i}M_{k,i}, & \text{ for }  k\ne 0  \text{ or }  n\ne 0; \\
0, &  \text{ if }  k=0 \text{ and } n=0.
\end{cases}
\end{displaymath}

The first rising $k$-binomial transforms are as follows:
\begin{align*}
R_{1}=&\left\{2, 4, 10, 26, 68, 178, \ldots \right\} : A052995-\{0\} \text{ or } A055819-\{1\}\\
R_{2}=&\left\{2, 6, 34, 198, 1154, 6726, \ldots \right\}\\
R_{3}=&\left\{2, 8, 86, 938, 10232,  \ldots \right\}\\
R_{4}=&\left\{2, 10, 178, 3194, 57314, \ldots \right\}\\
R_{5}=&\left\{2, 12, 322, 8682, 234092, \ldots \right\}
\end{align*}

\begin{lemma}\label{rbinomial_T}
For any integer $n\ge0$ and $k\ne0$,
$$r_{k,n} = \sum_{i=0}^{n}\binom{n}{i}k^{i} M_{k,i} = M_{k,2n}.$$
\end{lemma}
\begin{proof}
This identity coincides with Theorem 4.10 in \cite{YK}.
\end{proof}


\begin{theorem}\label{rbinomial_Ta}
The rising $k$-binomial transform of the modified $k$-Fibonacci-like sequence $R_{k}=\left(r_{k,n}\right)_{n\geq0}$ satisfies the recurrence relation
$$r_{k,n+1} = (k^{2}+2)r_{k,n} -  r_{k,n-1} ~\text{for}~n\ge1$$
with the initial conditions $r_{k,0} = 2$, $r_{k,1} = 2k+2$.
\end{theorem}
\begin{proof}
From the definition of the modified $k$-Fibonacci-like sequence, we obtain
\begin{align*}
M_{k,2n+2}=&kM_{k,2n+1}+M_{k,2n}\\
=&k\left(kM_{k,2n}+M_{k,2n-1}\right)+M_{k,2n}\\
=&(k^{2}+1)M_{k,2n}+kM_{k,2n-1}\\
=&(k^{2}+1)M_{k,2n}+M_{k,2n}-M_{k,2n-2}\\
=&(k^{2}+2)M_{k,2n}-M_{k,2n-2}.
\end{align*}
By Lemma \ref{rbinomial_T}, since $r_{k,n}=M_{k,2n}$, then we have
$$r_{k,n+1}=(k^{2}+2)r_{k,n}-r_{k,n-1}.$$
\end{proof}

Similarly, Binet's formula for the rising $k$-binomial transform of the modified $k$-Fibonacci-like sequence is the following:

\begin{theorem}
Binet's formula for the rising $k$-binomial transform of the modified $k$-Fibonacci-like sequence is given by
$$r_{k,n}= (2k+2)\frac{t_{1}^{n}-t_{2}^{n}}{t_{1}-t_{2}}- 2\frac{t_{1}^{n-1}-t_{2}^{n-1}}{t_{1}-t_{2}},$$
where $t_{1}$ and $t_{2}$ are the roots of the characteristic equation $x^{2}-(k^{2}+2)x+1=0$, and $t_{1}>t_{2}$.
\end{theorem}
\begin{proof}
The proof is same as that of the binomial transform, which is in Theorem \ref{binet_T}.
\end{proof}

Similarly, the generating function for the rising $k$-binomial transform of the modified $k$-Fibonacci-like sequence is
$$
r_{k}(x) = \frac{2-(2k^2 - 2k +2)x}{1-(k^{2}+2)x+x^{2}}.
$$

\section{The falling \texorpdfstring{$k$}{Lg}-binomial transform of the modified \texorpdfstring{$k$}{Lg}-Fibonacci-like sequence }

The falling $k$-binomial transform of the modified $k$-Fibonacci-like sequence $\left(M_{k,n}\right)_{n\geq0}$ is denoted by $F_{k}=\left(f_{k,n}\right)_{n\geq0}$ where
\begin{displaymath}
f_{k,n} =\begin{cases}
\sum_{i=0}^{n}\binom{n}{i} k^{n-i}M_{k,i}, & \text{ for }  k\ne 0  \text{ or }  n\ne 0; \\
0, &  \text{ if }  k=0 \text{ and } n=0.
\end{cases}
\end{displaymath}

The first falling $k$-binomial transforms are as follows:
\begin{align*}
F_{1}=&\left\{2, 4, 10, 26, 68, 178, \ldots \right\} : A052995-\{0\} \text{ or } A055819-\{1\}\\
F_{2}=&\left\{2, 6, 22, 90, 386, 1686, \ldots \right\}\\
F_{3}=&\left\{2, 8, 38, 206, 1208, 7370,  \ldots \right\}\\
F_{4}=&\left\{2, 10, 58, 386, 2834, 22042 \ldots \right\}\\
F_{5}=&\left\{2, 12, 82, 642, 5612, 52722 \ldots \right\}
\end{align*}

\begin{lemma}\label{fbinomial_T}
The falling $k$-binomial transform of the modified $k$-Fibonacci-like sequence satisfies the relation
$$f_{k,n+1}-kf_{k,n} = \sum_{i=0}^{n}\binom{n}{i} k^{n-i}M_{k,i+1}.$$
\end{lemma}
\begin{proof}
The proof is similar to the proof of Lemma \ref{binomial_T}. And, we obtain
\begin{align*}
f_{k,n+1} - kf_{k,n}&= \sum_{i=0}^{n+1} \binom{n+1}{i}k^{n+1-i} M_{k,i}-\sum_{i=0}^{n}\binom{n}{i} k^{n+1-i}M_{k,i}\\
&=\sum_{i=1}^{n}\left[\binom{n+1}{i}-\binom{n}{i} \right]k^{n+1-i}M_{k,i} + M_{k,n+1} \\
&=\sum_{i=1}^{n}\binom{n}{i-1}k^{n+1-i}M_{k,i}+ M_{k,n+1}\\
&=\sum_{i=0}^{n-1}\binom{n}{i}k^{n-i}M_{k,i+1} + \binom{n}{n}M_{k,n+1}=\sum_{i=0}^{n}\binom{n}{i}k^{n-i}M_{k,i+1}.
\end{align*}
\end{proof}
Note that $f_{k,n+1} = \sum_{i=0}^{n}\binom{n}{i}\left( k^{n+1-i}M_{k,i}+  k^{n-i}M_{k,i+1}\right)$.

\begin{theorem}\label{fbinomial_Ta}
The falling $k$-binomial transform of the modified $k$-Fibonacci-like sequence $F_{k}=\left(f_{k,n}\right)_{n\geq0}$ satisfies the recurrence relation
$$f_{k,n+1} = 3kf_{k,n} -  (2k^{2}-1)f_{k,n-1} ~\text{for}~n\ge1$$
with the initial conditions $f_{k,0} = 2$, $f_{k,1} = 2k+2$.
\end{theorem}
\begin{proof}
By Lemma \ref{fbinomial_T}, since $f_{k,n+1} = \sum_{i=0}^{n}\binom{n}{i}\left(k^{n+1-i}M_{k,i}+k^{n-i}M_{k,i+1}\right)$, then we have
\begin{align*}
f_{k,n+1}=&\sum_{i=0}^{n}\binom{n}{i}k^{n-i}\left(k M_{k,i}+M_{k,i+1}\right)\\
=&\sum_{i=1}^{n}\binom{n}{i}k^{n-i}\left(2kM_{k,i}+M_{k,i-1}\right)+k^{n}\left(k M_{k,0}+M_{k,1}\right)\\
=&2k\sum_{i=1}^{n}\binom{n}{i}k^{n-i}M_{k,i} +\sum_{i=1}^{n}\binom{n}{i}k^{n-i}M_{k,i-1} + k^{n}\left(kM_{k,0}+M_{k,1}\right)\\
=&2k\sum_{i=0}^{n}\binom{n}{i}k^{n-i}M_{k,i} +\sum_{i=1}^{n}\binom{n}{i}k^{n-i}M_{k,i-1}\\
& + k^{n}\left(kM_{k,0}+M_{k,1}-2kM_{k,0}\right)\\
=&2kf_{k,n} +\sum_{i=1}^{n}\binom{n}{i}k^{n-i}M_{k,i-1} + k^{n}\left(M_{k,1}-kM_{k,0}\right).
\end{align*}
On the other hand, in the case of $\binom{n-1}{n}=0$, we can obtain the following:
\begin{align*}
kf_{k,n} =&2k^2f_{k,n-1}+\sum_{i=1}^{n-1}\binom{n-1}{i}k^{n-i}M_{k,i-1}+k^{n}\left(M_{k,1}-kM_{k,0}\right)\\
=&2k^2f_{k,n-1}-\left[f_{k,n-1}-\sum_{i=0}^{n-1}\binom{n-1}{i}k^{n-1-i}M_{k,i}\right]\\
&+\sum_{i=0}^{n-2}\binom{n-1}{i+1}k^{n-1-i}M_{k,i}+k^{n}\left(M_{k,1}-kM_{k,0}\right)\\
=&\left(2k^{2}-1\right)f_{k,n-1}+\sum_{i=0}^{n-1}\left[\binom{n-1}{i}+\binom{n-1}{i+1}\right]k^{n-1-i}M_{k,i}\\
&+k^{n}\left(M_{k,1}-kM_{k,0}\right)\\
=&\left(2k^{2}-1\right)f_{k,n-1}+\sum_{i=0}^{n-1}\binom{n}{i+1}k^{n-1-i}M_{k,i}+k^{n}\left(M_{k,1}-kM_{k,0}\right)\\
=&\left(2k^{2}-1\right)f_{k,n-1}+\sum_{i=1}^{n}\binom{n}{i}k^{n-i}M_{k,i-1}+k^{n}\left(M_{k,1}-kM_{k,0}\right).
\end{align*}
Based on the above two identities, this study draws the below formulas.
$$f_{k,n+1}-2kf_{k,n}=kf_{k,n}-\left(2k^{2}-1\right)f_{k,n-1},$$
and so
$$f_{k,n+1} = 3kf_{k,n} - (2k^{2}-1) f_{k,n-1}.$$
\end{proof}

Similarly, Binet's formula for the falling $k$-binomial transform of the modified $k$-Fibonacci-like sequence is the following:

\begin{theorem}
Binet's formula for the falling $k$-binomial transform of the modified $k$-Fibonacci-like sequence is given by
$$f_{k,n}= (2k+2)\frac{u_{1}^{n}-u_{2}^{n}}{u_{1}-u_{2}}- 2\frac{u_{1}^{n-1}-u_{2}^{n-1}}{u_{1}-u_{2}},$$
where $u_{1}$ and $u_{2}$ are the roots of the characteristic equation $x^{2}-3kx+(2k^{2}-1)=0$, and $u_{1}>u_{2}$.
\end{theorem}
\begin{proof}
The proof is same as that of the binomial transform, which is in Theorem \ref{binet_T}.
\end{proof}

Similarly, the generating function for the falling $k$-binomial transform of the modified $k$-Fibonacci-like sequence is
$$
f_{k}(x) = \frac{2+(2-4k)x}{1-3kx+(2k^{2}-1)x^{2}}.
$$

\section{Conclusion}
This paper applies the four transforms- the binomial, $k$-binomial, rising $k$-binomial and falling $k$-binomial transforms- to the modified $k$-Fibonacci-like sequence. Although most of the results are rather similar to those of the previous sequences, this study is still meaningful as they introduce several new approaches and methods to derive the formulas. This study, furthermore, examines Binet's formulas and generating functions of the four transforms.




\begin{thebibliography}{}


\bibitem{BJS}
{P. Bhadouria, D. Jhala, B. Singh}, {Binomial Transforms of the $k$-Lucas Sequences and its Properties}, \textit{Journal of Mathematics and Computer Science}, \textbf{8} (2014), 81--92.

\bibitem{KC}
{K.W. Chen}, {Identities from the binomial transform}, \textit{Journal of Number Theory}, \textbf{124} (2007), 142--150.




\bibitem{FP01}
{S. Falc$\acute{\rm{o}}$n, $\acute{\rm{A}}$. Plaza}, {On the Fibonacci $k$-numbers}, \textit{Chaos, Solitons \& Fractals}, \textbf{32} (2007), 1615--1624.

\bibitem{FP02}
{S. Falc$\acute{\rm{o}}$n, $\acute{\rm{A}}$. Plaza}, {The $k$-Fibonacci sequence and the Pascal 2-triangle.}, \textit{Chaos, Solitons \& Fractals}, \textbf{33(1)}, (2007), 38-49.

\bibitem{FP03}
{S. Falc$\acute{\rm{o}}$n, $\acute{\rm{A}}$. Plaza}, {Binomial transforms of the $k$-Fibonacci sequence.}, \textit{International Journal of Nonlinear Sciences and Numerical Simulation}, \textbf{10(11-12)}, (2009), 1527-1538.

\bibitem{YK}
{Y. Kwon}, {A note on the modified $k$-Fibonacci-like sequence}, \textit{Communication of the Korean Mathematical Society}, \textbf{31} (2016), 1--16.


\bibitem{HP}
{H. Prodinger}, {Some information about the binomial transform}, \textit{The Fibonacci Quarterly}, \textbf{32(5)} (1994), 412--415.

\bibitem{S}
{N.J.A. Sloane}, {The On-Line Encyclopedia of Integer Sequences}, \url{https://oeis.org}.

\bibitem{SS}
{M.Z. Spivey, L.L. Steil}, {The $k$-Binomial Transform and the Hankel Transform}, \textit{Journal of Integer Sequences}, \textbf{9} (2006), 1--19.


\bibitem{YT02}
{N. Yilmaz, N. Taskara}, {Binomial transforms of the Padovan and Perrin matrix sequences}, \textit{Abstract and Applied Analysis}, \textit{Article number : 497418} (2013).

\end{thebibliography}


\section*{Acknowledgement}
The author thanks Ms. Juhee Son for proofreading this manuscript.

\end{document}